\documentclass{daj}
\usepackage{amsfonts,amsmath,latexsym,color,epsfig,hyperref,enumitem, amssymb}
\newtheorem{theorem}{Theorem}
\newtheorem{lemma}{Lemma}
\newtheorem{proposition}{Proposition}
\newtheorem{corollary}{Corollary}
\newtheorem{claim}{Claim}
\newcommand{\E}{\mathbb E}
\newcommand{\F}{\mathbb F}
\newenvironment{proof}
      {\medskip\noindent{\bf Proof:}\hspace{1mm}}
      {\hfill$\Box$\medskip}

\dajAUTHORdetails{%
  title = {Random multilinear maps and the Erd\H{o}s box problem}, 
  author = {David Conlon, Cosmin Pohoata, and Dmitriy Zakharov},
  plaintextauthor = {David Conlon, Cosmin Pohoata, Dmitriy Zakharov},
  plaintexttitle = {Random multilinear maps and the Erd\H{o}s box problem}, 
  runningtitle = {Random Multilinear Maps and the Erd\H{o}s Box Problem}, 
  copyrightauthor = {D. Conlon, C. Pohoata, and D. Zakharov},
  keywords = {Erd\H{o}s box problem, extremal numbers, Zarankiewicz problem, hypergraphs.},
}   

\dajEDITORdetails{%
   year={2021},
   number={17},
   received={19 November 2020},   
   published={28 September 2021},  
   doi={10.19086/da.28336},       
}   

\begin{document}

\begin{frontmatter}[classification=text]

\title{Random Multilinear Maps\\ and the Erd\H{o}s Box Problem} 

\author[david]{David Conlon\thanks{Research supported by NSF Award DMS-2054452.}}
\author[cosmin]{Cosmin Pohoata}
\author[dima]{Dmitriy Zakharov\thanks{Research supported by a grant of the Russian Government N 075-15-2019-1926.}}

\begin{abstract}
By using random multilinear maps, we provide new lower bounds for the Erd\H{o}s box problem, the problem of estimating the extremal number of the complete $d$-partite $d$-uniform hypergraph with two vertices in each part, thereby improving on work of Gunderson, R\"{o}dl and Sidorenko.
\end{abstract}
\end{frontmatter}

\section{Introduction}

Writing $K_{s_{1},\ldots,s_{d}}^{(d)}$ for the complete $d$-partite $d$-uniform hypergraph with parts of orders $s_{1},\ldots,s_{d}$, the extremal number $\operatorname{ex}_{d}(n,K_{s_{1},\ldots,s_{d}}^{(d)})$ is the maximum number of edges in a $d$-uniform hypergraph on $n$ vertices containing no copy of $K_{s_{1},\ldots,s_{d}}^{(d)}$. Already for $d = 2$, the problem of determining these extremal numbers is one of the most famous in combinatorics, known as the Zarankiewicz problem. 
The classic result on this problem, due to K\H{o}v\'ari, S\'os and Tur\'an \cite{KST54}, says that
$$\operatorname{ex}_2(n,K_{s_{1},s_{2}}) = O\left(n^{2-1/s_{1}}\right)$$
for all $s_1 \leq s_2$. However, this upper bound has only been matched by a construction with $\Omega(n^{2-1/s_{1}})$ edges when $s_2 > (s_1-1)!$, a result which, in this concise form, is due to Alon, Koll\'ar, R\'onyai and Szab\'o~\cite{ARSz99, KRSz96}, but builds on a long history of earlier work on special cases (see, for example, the comprehensive survey~\cite{FS13}).

Generalizing the K\H{o}v\'ari--S\'os--Tur\'an bound, Erd\H{o}s \cite{Erd64} showed that
\begin{equation} \label{erdos}
\operatorname{ex}_{d}(n,K_{s_{1},\ldots,s_{d}}^{(d)}) = O\left(n^{d - \frac{1}{s_{1}\cdots s_{d-1}}}\right)
\end{equation}
for all $s_{1} \leq s_{2} \leq \ldots \leq s_{d}$. An analogue of the Alon--Koll\'ar--R\'onyai--Szab\'o result, due to Ma, Yuan and Zhang \cite{MYZ18}, is also known in this context and says that~\eqref{erdos} is tight up to the constant provided that $s_{d}$ is sufficiently large in terms of $s_{1},\ldots,s_{d-1}$. The proof of this result is based on an application of the random algebraic method, introduced by Bukh \cite{Buk15} and further developed in \cite{BC18} and \cite{Con19}.

Our concern then will be with determining the value of $\operatorname{ex}_{d}(n,K_{s_{1},\ldots,s_{d}}^{(d)})$ in the particular case when $s_{1}=\cdots=s_{d}=2$. In the literature, this problem, originating in the work of Erd\H{o}s~\cite{Erd64}, is sometimes referred to as the \emph{box problem}, owing to a simple reformulation in terms of finding the largest subset of the grid $\{1, 2, \dots, n\}^d$ which does not contain the vertices of a $d$-dimensional box (see also~\cite{KKM02} for a connection to a problem in analysis). By~\eqref{erdos}, we have
\begin{equation} \label{erdos2}
\operatorname{ex}_{d}(n,K_{2,\ldots,2}^{(d)}) = O\left(n^{d - \frac{1}{2^{d-1}}}\right).
\end{equation}
While in the case $d=2$ it has long been known that $\operatorname{ex}_2(n,K_{2,2})=\Theta(n^{3/2})$, with a matching construction due to Klein~\cite{Erd38} even predating the K\H{o}v\'ari--S\'os--Tur\'an bound, there has been very little success in finding constructions matching~\eqref{erdos2} for $d \geq 3$. Indeed, it is unclear whether they should even exist. For $d=3$, the best available construction is due to Katz, Krop and Maggioni~\cite{KKM02}, who showed that $\operatorname{ex}_{3}(n,K_{2,2,2}^{(3)}) =\Omega(n^{8/3})$. For general $d$, there is a simple, but longstanding, lower bound
\begin{equation} \label{deletion}\operatorname{ex}_{d}(n,K_{2,\ldots,2}^{(d)}) = \Omega\left(n^{d - \frac{d}{2^d-1}}\right)
\end{equation}
coming from an application of the probabilistic deletion method. Besides the Katz--Krop--Maggioni construction, the only improvement to this bound is an elegant construction of Gunderson, R\"{o}dl and Sidorenko~\cite{GRS99}, which amplified the deletion argument by introducing algebraic structure on one of the sides of the $d$-partition and using random hyperplanes to define the edges.

\begin{theorem}[Gunderson--R\"{o}dl--Sidorenko] \label{GRS}
For any $d \geq 2$, let $s=s(d)$ be the smallest positive integer $s$ (if it exists) such that $(sd-1)/(2^{d}-1)$ is an integer. Then
\begin{equation*}
\operatorname{ex}_{d}(n,K_{2,\ldots,2}^{(d)}) = \Omega\left(n^{d - \frac{d-1/s}{2^d-1}}\right).
\end{equation*}
\end{theorem}

It is easy to see that the number $s=s(d)$ exists precisely when $d$ and $2^{d}-1$ are relatively prime, which holds, for instance, when $d$ is a prime number or a power of $2$, but does not hold for many other numbers, such as $d=6$, $12$, $18$, $20$, $21$. In fact, their result fails to apply for a positive proportion of the positive integers, as may be seen by noting that if the condition $(d, 2^d - 1) = 1$ fails for a given $d$, then it also fails for all multiples of $d$.

In this paper, we improve on the lower bound from Theorem \ref{GRS} by establishing the following result, whose proof refines the method from \cite{GRS99} by introducing algebraic structure on each side of the $d$-partition and using random multilinear maps to define the edges.

\begin{theorem}\label{main}
For any $d \geq 2$, let $r$ and $s$ be positive integers such that $d(s-1) < (2^d - 1)r$. Then
$$
\operatorname{ex}_{d}(n,K_{2,\ldots,2}^{(d)}) = \Omega\left(n^{d - \frac{r}{s}}\right).
$$
\end{theorem}

This not only improves the lower bound for the 
box problem provided by Theorem \ref{GRS} for any $d$ which is not a power of $2$, but it also yields a gain over the probabilistic deletion bound \eqref{deletion} for {\textit{all}} uniformities $d$. To see this, note that if $d \geq 2$, then $d$ never divides $2^{d}-1$, so we may set $r = 1$ and $s = \lceil \frac{2^d-1}{d}\rceil > \frac{2^d-1}{d}$.

\begin{corollary} \label{new}
For any $d \geq 2$,
$$
{\rm ex}_d(n, K_{2, \ldots, 2}^{(d)}) = \Omega\left( n^{d - \lceil \frac{2^d-1}{d}\rceil^{-1}} \right).
$$
\end{corollary}

For the reader's convenience, we include below a table comparing the bounds provided by the deletion bound \eqref{deletion}, by Gunderson, R\"odl and Sidorenko's Theorem~\ref{GRS} and by our Corollary~\ref{new}. A number $\alpha$ on the $d^{\textrm{th}}$ row of the table means that the corresponding method gives the lower bound
$$
{\rm ex}_d(n, K_{2, \ldots, 2}) = \Omega\left(n^{d - 1/\alpha}\right),
$$
while an empty cell in the GRS column means that the method does not apply for that value of $d$. In particular, we note that our method recovers both the fact that $\operatorname{ex}(n,K_{2,2})=\Theta(n^{3/2})$ and the lower bound $\operatorname{ex}_{3}(n,K_{2,2,2}^{(3)}) =\Omega(n^{8/3})$ of Katz, Krop and Maggioni.

\begin{center}
\begin{tabular}{ |r|r|r|r| } 
 \hline
 $d$ & Deletion & GRS & Corollary \ref{new}\\
 \hline
 2  &  1.50  &  2.00  &  2.00 \\
3  &  2.33  &  2.50  &  3.00 \\
4  &  3.75  &  4.00 &  4.00 \\
5  &  6.20  &  6.25  &  7.00 \\
6  &  10.50  & &  11.00 \\
7  &  18.14  &  18.16  &  19.00 \\
8  &  31.87  &  32.00  &  32.00 \\
9  &  56.77  &  56.80  &  57.00 \\
10  &  102.30  &  102.33  &  103.00 \\
11  &  186.09  &  186.10  &  187.00 \\
12  &  341.25  & &  342.00 \\
13  &  630.07  &  630.08  &  631.00 \\
14  &  1170.21  &  1170.22  &  1171.00 \\
15  &  2184.46  &  2184.50  &  2185.00 \\
16  &  4095.93  &  4096.00  &  4096.00 \\
17  &  7710.05  &  7710.06  &  7711.00 \\
18  &  14563.50  & &  14564.00 \\
19  &  27594.05  &  27594.05  &  27595.00 \\
20  &  52428.75  & &  52429.00 \\
21  &  99864.33  & &  99865.00 \\
22  &  190650.13  &  190650.14  &  190651.00 \\
 \hline
\end{tabular}
\end{center}

By a result of Ferber, McKinley and Samotij~\cite[Theorem 9]{FMS20}, any polynomial gain over the deletion lower bound for the extremal number of a uniform hypergraph $\mathcal H$ implies an optimal counting result for the number of $\mathcal H$-free graphs on $n$ vertices. In combination with Corollary~\ref{new}, this implies the following result, generalizing a celebrated theorem of Kleitman and Winston~\cite{KW82} on the $d = 2$ case.

\begin{corollary} \label{counting}
For any $d \geq 2$, let $\mathcal{F}_{n}\left(K_{2,\ldots,2}^{(d)}\right)$ be the set of all (labeled) $K_{2,\ldots,2}^{(d)}$-free $d$-uniform hypergraphs with vertex set $\left\{1,\ldots,n\right\}$. Then there exists a positive constant $C$ depending only on $d$ and an infinite sequence of positive integers $n$ for which
$$\left|\mathcal{F}_{n}\left(K_{2,\ldots,2}^{(d)}\right)\right| \leq 2^{C \cdot {\rm ex}_d(n, K_{2, \ldots, 2}^{(d)})}.$$
\end{corollary}


\section{New lower bounds for the Erd\H{o}s box problem}

\subsection{Linear algebra preliminaries}

Let $V_1, \ldots, V_d$ be finite-dimensional vector spaces over the field $\F_q$. Following standard convention, we call a function $T: V_1 \times \cdots \times V_d \rightarrow \F_q$ {\it multilinear} if, for every $i \in \left\{1,\ldots,d\right\}$ and every fixed choice of $x_{j} \in V_{j}$ for each $j \neq i$, the function $T\left(x_{1},\ldots,x_{i-1},x,x_{i+1},\ldots,x_{d}\right)$, considered as a function on $V_{i}$, is linear over $\mathbb{F}_{q}$. 

The vector space of all multilinear functions $T: V_1 \times \cdots \times V_d \rightarrow \F_q$ can be naturally identified with the space $V_1^* \otimes \cdots \otimes V_d^*$, where $V^*$ denotes the dual space of $V$. A {\it uniformly random} multilinear function $T: V_1 \times \cdots \times V_d \rightarrow \F_q$ is then a random element of the space $V_1^* \otimes \cdots \otimes V_d^*$, chosen according to the uniform distribution. 

If, for each $i$, we have a subspace $U_i \subset V_i$, then we can define a restriction map
$$
r: V_1^* \otimes \cdots \otimes V_d^* \rightarrow U_1^* \otimes \cdots \otimes U_d^*.
$$

We have the following simple, but important, claim about these restriction maps.

\begin{claim}\label{triv}
The restriction $r(T)$ of a uniformly random multilinear function $T$ is again uniformly random.
\end{claim}

\begin{proof} 
The map $r$ is linear and surjective and so all $T' \in U_1^* \otimes \cdots \otimes U_d^*$ have the same number of preimages in $V_1^* \otimes \cdots \otimes V_d^*$.
\end{proof}

It will also be useful to note the following simple consequence of multilinearity.

\begin{proposition}\label{bra}
Suppose that $T: V_1 \times \cdots \times V_d \rightarrow \F_q$ is multilinear and, for every $i = 1, \ldots, d$, there are vectors $v_i^0, v_i^1 \in V_i$ such that
$$
T(v_1^{\varepsilon_1}, \ldots, v_d^{\varepsilon_d}) = 1
$$
for all $2^d$ choices of $\varepsilon_i \in \{0, 1\}$. Then, for any $u_i$ which lie in the affine hull of $v_i^0, v_i^1$ for each $i = 1, \ldots, d$,
$$
T(u_1, \ldots, u_d) = 1.
$$
\end{proposition}

\begin{proof} Write $u_i = \alpha_i^0 v_i^0 + \alpha_i^1 v_i^1$ for some $\alpha_i^0 + \alpha_i^1 = 1$. Then, by multilinearity, we have
\begin{align*}
T(u_1, \ldots, u_d) &= \sum_{\varepsilon_1, \ldots, \varepsilon_d \in \{0, 1\}} \alpha_1^{\varepsilon_1} \cdots \alpha_d^{\varepsilon_d} T(v_1^{\varepsilon_1}, \ldots, v_d^{\varepsilon_d})\\
&= \sum_{\varepsilon_1, \ldots, \varepsilon_d \in \{0, 1\}} \alpha_1^{\varepsilon_1} \cdots \alpha_d^{\varepsilon_d} = (\alpha_1^0 + \alpha_1^1) \cdots (\alpha_d^0 + \alpha_d^1) = 1,
\end{align*}
as required.
\end{proof}

\subsection{Proof of Theorem \ref{main}}

Fix positive integers $d$, $r$ and $s$ and let $q$ be a large prime power. Let $V = \F_q^s$ and
let $T_1, \ldots, T_r \in V^{*\otimes d}$ be independent uniformly random multilinear functions. Let $\mathcal{H}$ be the $d$-partite $d$-uniform hypergraph between $d$ copies of $V$ whose edge set $\mathcal{E}$ consists of all tuples $(v_1, \ldots, v_d) \in V^d$ such that $T_i(v_1, \ldots, v_d) = 1$ for all $i = 1, \ldots, r$. Let us estimate the expected number of edges in $\mathcal{H}$.

\begin{claim}\label{exp}
$\mathbb{E} \left[|\mathcal{E}|\right] = (q^s - 1)^d q^{-r} \sim q^{ds-r}$.
\end{claim}

\begin{proof} Note that if one of $v_1, \ldots, v_d$ is zero, then $T_i(v_1, \ldots, v_d) = 0$, so we may assume that $(v_1, \ldots, v_d)$ is one of the $(q^s - 1)^d$ remaining sequences of non-zero vectors and calculate the probability that it belongs to $\mathcal E$. Let $U_i = \langle v_i \rangle \subset V$, a one-dimensional subspace of $V$. By Claim \ref{triv}, the restriction $T'_i$ of $T_i$ to $U_1 \times \cdots \times U_d$ is uniformly distributed in $U_1^* \otimes \cdots \otimes U_d^*$. But the latter space is one-dimensional and so $T'_i(v_1, \ldots, v_d)$ takes the value $1$ with probability $q^{-1}$. Since $T_1, \ldots, T_r$ are independent, the functions $T'_1, \ldots, T'_r$ are independent, so they all are equal to one at $(v_1, \ldots, v_d)$ with probability exactly $q^{-r}$.  
\end{proof}

We now estimate the expected number of (appropriately ordered) copies of $K_{2,\ldots,2}^{(d)}$ in $\mathcal{H}$.

\begin{claim}\label{cl}
Let $\mathcal F$ denote the family of all $(v_1^0, v_1^1, \ldots, v_d^0, v_d^1) \in V^{2d}$ where $v_j^0 \neq v_j^1$ for all $j$ and $T_i(v_1^{\varepsilon_1}, \ldots, v_d^{\varepsilon_d}) = 1$ for all $i = 1, \ldots, r$ and all choices of $\varepsilon_1, \ldots, \varepsilon_d \in \{0, 1\}$. Then $\E[|\mathcal F|] \sim q^{2ds - 2^d r}$.
\end{claim}

\begin{proof} If, for some $j = 1, \ldots, d$, the vectors $v_j^0$ and $v_j^1$ are collinear, say $v_j^1 = \lambda v_j^0$ for some $\lambda \neq 1$ (but allowing $\lambda = 0$), then 
$$
T(v_1^0, \ldots, v_j^1, \ldots, v_d^0) = \lambda T(v_1^0, \ldots, v_j^0, \ldots, v_d^0),
$$
so these two numbers cannot be equal to 1 simultaneously. Therefore, we may restrict attention to only those tuples where $v_j^0$ and $v_j^1$ are linearly independent for all $j = 1, \ldots, d$. 

Fix one of the $(q^s - 1)^d(q^s - q)^d$ remaining tuples $\bar v = (v_1^0, v_1^1, \ldots, v_d^0, v_d^1)$ and let us compute the probability that $\bar v \in \mathcal F$. Let $U_j = \langle v_j^0, v_j^1\rangle$ be the two-dimensional vector space spanned by $v_j^0$ and $v_j^1$. By Claim \ref{triv}, the restriction $T'_i$ of $T_i$ to $U_1 \times \cdots \times U_d$ is uniformly distributed in $U_1^* \otimes \cdots \otimes U_d^*$. Moreover, the independence of $T_1, \ldots, T_r$ implies that $T'_1, \ldots, T'_r$ are also independent. Now observe that the set of $2^d$ tensors
$$
\{v_1^{\varepsilon_1} \otimes \cdots \otimes v_d^{\varepsilon_d} :~\varepsilon_j \in \{0, 1\}\}
$$
forms a basis for the space $U_1 \otimes \cdots \otimes U_d$. Therefore, there exists a unique $R \in U_1^* \otimes \cdots \otimes U_d^*$ such that $R(v_1^{\varepsilon_1}, \ldots, v_d^{\varepsilon_d}) = 1$ for all $\varepsilon_j \in \{0, 1\}$. Moreover, since there are $q^{2^d}$ different choices for the value of a function in $U_1^* \otimes \cdots \otimes U_d^*$ at the $(v_1^{\varepsilon_1}, \ldots, v_d^{\varepsilon_d})$ and each such choice determines a unique function, the probability that $T'_i = R$ is $q^{-2^d}$. Since $\bar v \in \mathcal F$ if and only if $T'_i = R$ for all $i = 1, \ldots, r$, the independence of the $T'_i$ implies that the probability $\bar v \in \mathcal F$ is $q^{-2^d r}$. Thus, 
$$
\E[|\mathcal F|] = (q^s - 1)^d(q^s - q)^d q^{-2^d r} \sim q^{2ds - 2^d r},
$$
as required.
\end{proof}

The next step is crucial.

\begin{lemma}\label{crux}
Let $\mathcal B$ be the family of all $(v_1, \ldots, v_d) \in \mathcal E$ for which there exists $(v'_1, \ldots, v'_d) \in V^d$ such that $(v_1, v'_1, \ldots, v_d, v'_d) \in \mathcal F$. Then 
$$
\E[|\mathcal B|] \le (1+o(1)) q^{-d}\E[|\mathcal F|].
$$
\end{lemma}

\begin{proof} Given a sequence of affine lines $l_1, \ldots, l_d \subset V$, denote by $P(l_1, \ldots, l_d)$ the set of all sequences $(x_1, x'_1, \ldots, x_d, x'_d) \in V^{2d}$ such that $x_j$ and $x'_j$ are distinct and lie on $l_j$ for all $j$. Clearly,
$$
|P(l_1, \ldots, l_d)| = q^d (q-1)^d \sim q^{2d}.
$$
Note that:
\begin{enumerate}
    \item If $(l_1, \ldots, l_d) \neq (l'_1, \ldots, l'_d)$, then
    $P(l_1, \ldots, l_d) \cap P(l'_1, \ldots, l'_d) = \emptyset$, 
    since the lines $l_1, \ldots, l_d$ are uniquely determined by any member of $P(l_1, \ldots, l_d)$.
    \item If $P(l_1, \ldots, l_d) \cap \mathcal F \neq \emptyset$, then $P(l_1, \ldots, l_d) \subset \mathcal F$ by Proposition \ref{bra}.
    \item Any $\bar v \in \mathcal F$ is contained in $P(l_1, \ldots, l_d)$ for some $l_1, \ldots, l_d$.
\end{enumerate}

Denote the family of all tuples $(l_1, \ldots, l_d)$ such that $P(l_1, \ldots, l_d) \cap \mathcal F \neq \emptyset$ by $\mathcal L$. By the observations above, we have that 
$$
|\mathcal L| q^d (q-1)^d = |\mathcal F|.
$$
On the other hand, it is clear that
$$
\mathcal B = \bigcup_{(l_1, \ldots, l_d) \in \mathcal L} l_1 \times l_2 \times \ldots \times l_d,
$$
so that
$$
|\mathcal B| \le q^d |\mathcal L| = (q-1)^{-d} |\mathcal F|.
$$
Taking expectations, we obtain the required result.
\end{proof}

By definition, the subgraph $\mathcal{H}'$ of $\mathcal{H}$ with edge set $\mathcal E \setminus \mathcal B$ is $K_{2, \ldots, 2}$-free. 
By Lemma \ref{crux} and Claim \ref{cl}, 
$$
\E[|\mathcal B|] \le (1+o(1))q^{-d}\E[|\mathcal F|] = (1+o(1))q^{2ds-2^d r - d}.
$$
On the other hand, by Claim \ref{exp}, $\E[|\mathcal E|] \sim q^{ds - r}$. By the assumption on $r$ and $s$ from the statement of Theorem \ref{main}, we have 
$$
    2ds - 2^d r -d < ds - r,
$$
which immediately implies that $\E[|\mathcal B|] = o(\E[|\mathcal{E}|])$. Therefore, there must exist a $K_{2, \ldots, 2}^{(d)}$-free hypergraph $\mathcal H'$ on a ground set of size $n = d q^s$ with edge set $\mathcal{E}'$ satisfying
$$
|\mathcal E'| = (1+o(1)) q^{ds-r} = (c+o(1))n^{d - \frac{r}{s}},
$$
where $c = d^{\frac{r}{s} -d}$, completing the proof of Theorem \ref{main}.



\bibliographystyle{amsplain}

\begin{thebibliography}{99}

\bibitem{ARSz99}
N. Alon, L. R\'onyai and T. Szab\'o, Norm-graphs: variations and applications, {\it J. Combin. Theory Ser. B} {\bf 76} (1999), 280--290.

\bibitem{Buk15} 
B. Bukh, Random algebraic construction of extremal graphs, {\it{Bull. London Math. Soc.}} {\bf{47}} (2015), 939--945.



\bibitem{BC18} B. Bukh and D. Conlon, Rational exponents in extremal graph theory, {\it{J. Eur. Math. Soc.}} {\bf{20}} (2018), 1747--1757.

\bibitem{Con19} D. Conlon, Graphs with few paths of prescribed length between any two vertices, {\it{Bull. Lond. Math. Soc.}} {\bf{51}} (2019), 1015--1021.

\bibitem{Erd38}
{P. Erd\H{o}s,} {On sequences of integers no one of which divides the product of two others and on some related problems,} {\it Mitt. Forsch.-Inst. Math. Mech. Univ. Tomsk} {\bf 2} (1938), 74--82.

\bibitem{Erd64} P. Erd\H{o}s, On extremal problems of graphs and generalized hypergraphs, {\it{Israel J. Math.}} {\bf{2}} (1964), 183--190.


\bibitem{FMS20} A. Ferber, G. McKinley and W. Samotij, Supersaturated sparse graphs and hypergraphs, {\textit{Int. Math. Res. Not. IMRN}} {\bf 2020} (2020), 378--402.

\bibitem{FS13}
Z. F\"uredi and M. Simonovits, The history of degenerate (bipartite) extremal graph problems, in Erd\H{o}s centennial, 169--264, Bolyai Soc. Math. Stud., 25, J\'anos Bolyai Math. Soc., Budapest, 2013. 

\bibitem{GRS99} D. S. Gunderson, V. R\"{o}dl and A. Sidorenko, Extremal problems for sets forming
Boolean algebras and complete partite hypergraphs, {\textit{J. Combin. Theory Ser. A}} {\textbf{88}}
(1999), 342--367.

\bibitem{KKM02} N. Katz, E. Krop and M. Maggioni, Remarks on the box problem, {\it{Math. Res. Lett.}} {\bf{9}} (2002), 515--519.

\bibitem{KRSz96} 
J. Koll\'ar, L. R\'onyai and T. Szab\'o, Norm-graphs and bipartite Turán numbers, {\it Combinatorica} {\bf 16} (1996), 399--406.

\bibitem{KST54} T. K\H{o}v\'ari, V. T. S\'os and P. Tur\'an, On a problem of K. Zarankiewicz, {\it{Colloq.
Math.}} {\bf{3}} (1954), 50--57.

\bibitem{KW82} D. Kleitman and D. Wilson, On the number of graphs without $4$-cycles, {\it{Discrete Math.}} {\bf{41}} (1982), 167--172.
 
\bibitem{MYZ18} J. Ma, X. Yuan and M. Zhang, Some extremal results on complete degenerate hypergraphs, {\it{J. Combin. Theory Ser. A}} {\bf{154}} (2018), 598--609.

\end{thebibliography}


\begin{dajauthors}
\begin{authorinfo}[david]
  David Conlon\\
  Department of Mathematics\\
  California Institute of Technology\\
  Pasadena, USA\\
  dconlon@caltech.edu \\
  \url{http://www.its.caltech.edu/~dconlon/}
\end{authorinfo}
\begin{authorinfo}[cosmin]
  Cosmin Pohoata\\
  Department of Mathematics\\
  Yale University\\
  New Haven, USA\\
  andrei.pohoata@yale.edu\\
  \url{https://pohoatza.wordpress.com/}
\end{authorinfo}
\begin{authorinfo}[dima]
 Dmitriy Zakharov\\
  Laboratory of Combinatorial and Geometric Structures\\
  MIPT\\
  Moscow, Russia\\
 zakharov2k@gmail.com\\
  \url{https://combgeo.org/en/members/dmitriy-zakharov/}
\end{authorinfo}
\end{dajauthors}

\end{document}